\documentclass[12pt]{article}%

\usepackage{amsmath,enumerate}
\usepackage{amsfonts}
\usepackage{amssymb}
\usepackage{color}

\newtheorem{theorem}{Theorem}[section]

\newtheorem{corollary}[theorem]{Corollary}

\newtheorem{lemma}[theorem]{Lemma}

\newenvironment{proof}[1][Proof]{\noindent\textbf{#1.} }
{\hfill \ \rule{0.5em}{0.5em}}

\begin{document}

\title{Small dense subgraphs of polarity graphs and the extremal number for the 4-cycle}
\author{Michael Tait\thanks{Department of Mathematics, University of California San Diego, \texttt{mtait@math.ucsd.edu}} \and Craig Timmons\thanks{Department of Mathematics and Statistics, California State University Sacramento, \texttt{craig.timmons@csus.edu} }}
\date{}
\maketitle
\vspace{-5mm}
\begin{abstract}
In this note, we show that for any $m \in \{1,2, \dots , q +1 \}$, if $G$ is a polarity graph of a projective plane of order $q$ that has an oval, then $G$ contains a subgraph on $m + \binom{m}{2}$ vertices with  
$m^2+\frac{m^4}{8q} - O (  \frac{m^4}{q^{3/2} } +m )$ edges. 
As an application, we give the best known lower bounds on the Tur\'{a}n number $\mathrm{ex}(n, C_4)$ for certain values of $n$. In particular, we disprove a conjecture of Abreu, Balbuena, and Labbate concerning $\mathrm{ex}(q^2-q-2, C_4)$ where $q$ is a power of $2$.
\end{abstract}


\section{Introduction}

Let $F$ be a graph.  A graph $G$ is said to be $F$-\emph{free} if $G$ does not contain $F$ as a subgraph.  Let $\textup{ex}(n , F)$ denote the \emph{Tur\'{a}n number} of $F$, which is the maximum number of edges in an $n$-vertex $F$-free graph.  Write $\textup{Ex}(n  , F)$ for the family of $n$-vertex graphs that are $F$-free and have $\textup{ex}(n , F)$ edges.  Graphs in the family $\textup{Ex}(n , F)$ are called \emph{extremal graphs}.  Determining $\textup{ex}( n , F)$ for different graphs $F$ is one of the most well-studied problems in extremal graph theory.  A case of particular interest is when $F = C_4$, the cycle on four vertices.  A well known result of K\H{o}vari, S\'{o}s, and Tur\'{a}n \cite{kst} implies that $\textup{ex}(n , C_4) \leq \frac{1}{2} n^{3/2} + \frac{1}{2}n$.  Brown \cite{b}, and Erd\H{o}s, R\'{e}nyi, and S\'{o}s \cite{ers} proved that 
$\textup{ex}( q^2 + q + 1 , C_4) \geq \frac{1}{2} q (q + 1)^2$ whenever $q$ is a power of a prime.  It follows that 
$\textup{ex}(n , C_4) = \frac{1}{2} n^{3/2} + o( n^{3/2} )$.  For more on Tur\'{a}n numbers of bipartite graphs, we recommend the survey of F\"{u}redi and Simonovits \cite{fs}.  

The $C_4$-free graphs constructed in \cite{b} and \cite{ers} are examples of polarity graphs.  To define these graphs, we introduce some ideas from finite geometry.  Let $\mathcal{P}$ and $\mathcal{L}$ be disjoint, finite sets, and let $\mathcal{I}\subset \mathcal{P}\times \mathcal{L}$. We call the triple $(\mathcal{P}, \mathcal{L}, \mathcal{I})$ a {\em finite geometry}.  The elements of $\mathcal{P}$ are called {\em points}, and the elements of $\mathcal{L}$ are called {\em lines}.   A {\em polarity} of the geometry is a bijection from $\mathcal{P}\cup \mathcal{L}$ to $\mathcal{P}\cup \mathcal{L}$ that sends points to lines, sends lines to points, is an involution, and respects the incidence structure.  Given a finite geometry $(\mathcal{P}, \mathcal{L}, \mathcal{I})$ and a polarity $\pi$, the {\em polarity graph} $G_\pi$ is the graph with vertex set $V(G_\pi) = \mathcal{P}$ and edge set
\[
E(G_\pi) = \{\{p,q\}: p,q\in \mathcal{P}, (p, \pi(q))\in \mathcal{I}\}.
\]
\medskip
Note that $G_\pi$ will have loops if there is a point $p$ such that $(p, \pi(p))\in \mathcal{I}$. 
Such a point is called an {\em absolute point}. We will work with polarity graphs that have loops, and graphs obtained from polarity graphs by removing the loops.  A case of particular interest is when the geometry is 
the Desarguesian projective plane $PG(2,q)$.   For a prime power $q$, this is the plane obtained by considering the one-dimensional subspaces of $\mathbb{F}_q^3$ as points, the two-dimensional subspaces as lines, and incidence is defined by inclusion. A polarity of $PG(2,q)$ is given by sending points and lines to their orthogonal complements. The polarity graph obtained from $PG(2,q)$ with this polarity is often called the {\em Erd\H{o}s-R\'{e}nyi orthogonal polarity graph} and is denoted $ER_q$.  This is the graph that was constructed in \cite{b, ers} and we recommend \cite{bs} for a detailed study of this graph.  

Our main theorem will apply to $ER_q$ as well as to other polarity graphs that come from projective planes that contain an oval.  
An \emph{oval} in a projective plane of order $q$ is a set of $q+1$ points, no three of which are collinear.  It is known that $PG(2,q)$ always contains ovals.  One example is the set of $q+1$ points 
\[
\{(1 , t , t^2 ) : t \in \mathbb{F}_q \} \cup \{ (0,1,0) \}
\]
which form an oval in $PG(2,q)$.  There are also non-Desaurgesian planes that contain ovals.  We now state our main theorem.  

\begin{theorem}\label{dense subgraph}
Let $\Pi$ be a projective plane of order $q$ that contains an oval and has a polarity $\pi$.  If $m \in \{1,2, \dots , q + 1 \}$, then the polarity graph $G_{ \pi}$ contains a subgraph on at most $m + \binom{m}{2}$ vertices that has at least 
\[
2 \binom{m}{2} + \frac{m^4}{8q} - O \left( \frac{m^4}{q^{3/2} } + m \right)
\]
edges.  
\end{theorem}

Theorem \ref{dense subgraph} allows us to obtain the best-known lower bounds for $\mathrm{ex}(n, C_4)$ for certain values of $n$ by taking the graph $ER_q$ and removing a small subgraph that has many edges.  All of the best lower bounds in the current literature are obtained using this technique (see \cite{abl, fknw, tt}).  An open conjecture of McCuaig is that 
any graph in $\textup{Ex}(n , C_4)$ is an induced subgraph of some orthogonal polarity graph (cf \cite{fu1994}).  For $q \geq 15$ a prime power, F\"{u}redi \cite{fu1996} proved that any graph in $\textup{Ex}( q^2 + q + 1 , C_4)$ is an orthogonal polarity graph of some projective plane of order $q$.  For some recent progress on this problem, see \cite{fknw}.
By considering certain induced subgraphs of $ER_q$, Abreu, Balbuena, and Labbate \cite{abl} proved that
\[
\mathrm{ex}(q^2-q-2, C_4) \geq \frac{1}{2}q^3 - q^2
\]
whenever $q$ is a power of 2.  They conjectured that this lower bound is best possible.  
Using Theorem \ref{dense subgraph}, we answer their conjecture in the negative.

\begin{corollary}\label{abl conjecture}
If $q$ is a prime power, then  
\[
\mathrm{ex}(q^2 - q - 2, C_4) \geq \frac{1}{2}q^3 - q^2 + \frac{3}{2}q - O\left(q^{1/2}\right).
\]
\end{corollary}

Corollary \ref{abl conjecture} also improves the main result of \cite{tt}.  
In Section \ref{preliminaries} we give some necessary background on projective planes and polarity graphs. We prove 
Theorem \ref{dense subgraph} and Corollary \ref{abl conjecture} in Section \ref{proofs}.  We finish with some concluding remarks in Section \ref{conclusion}.


\section{Preliminaries}\label{preliminaries}

Let $\Pi = ( \mathcal{P} , \mathcal{L} , \mathcal{I} )$ be a finite projective plane of order $q$.  A {\em $k$-arc} is a set of $k$ points 
in $\Pi$ such that no three of the points are collinear.  It is known that $k \leq q +1$ when $q$ is odd, and $k \leq q + 2$ when $q$ is even.  A line $l \in \mathcal{L}$ is called {\em exterior}, {\em tangent}, or {\em secant} if it intersects the $k$-arc in 0, 1, or 2 points, respectively.  A $k$-arc has exactly $\binom{q}{2} + \binom{q + 2 - k }{2}$ exterior lines, $k ( q+ 2 - k )$ tangents, and $\binom{k}{2}$ secants (see \cite{demb}, page 147).  A $(q+1)$-arc is called an {\em oval} and in the plane $PG(2,q)$, ovals always exist (see \cite{demb}, Ch 1).  The next lemma is known.  A short proof is included for completeness.  

\begin{lemma}\label{eigenvalues of polarity graph}
Let $G$ be a polarity graph obtained from a projective plane of order $q$.  If $A$ is the adjacency matrix of $G$, then the eigenvalues of $A$ are $q+1$ and $\pm \sqrt{q}$.
\end{lemma}

\begin{proof}
In a projective plane, every pair of points is contained in a unique line. Therefore, in a polarity graph, there is a unique path of length $2$ between any pair of vertices (this path may include a loop). This means that $(A^2)_{ij} = 1$ whenever $i\not=j$. Since any point is on exactly $q+1$ lines, every vertex of $G$ has degree exactly $q+1$ where loops add 1 to the degree of a vertex. The diagonal entries of $A^2$ are all $q+1$ thus,
\[
A^2 = J + qI.
\]
The eigenvalues of $J+qI$ are $(q+1)^2$ with multiplicity 1, and $q$ with multiplicity $q^2 + q$.  
\end{proof}

We remark here that the multiplicity of $q+1$ is $1$ and the multiplicities of $\pm \sqrt{q}$ are such that the sum of the eigenvalues is the trace of $A$, which is the number of absolute points of $G$. This implies that given two polarity graphs from projective planes of order $q$, if they have the same number of absolute points, then they are cospectral. Since not all polarity graphs with the same number of absolute points are isomorphic, this gives examples of graphs that are not determined by their spectrum, which may be of independent interest. For more information about determining graphs by their spectrum, see \cite{haemers}.
\medskip

The next result is a consequence of Lemma \ref{eigenvalues of polarity graph} and the so-called Expander Mixing Lemma.
We provide a proof which uses some basic ideas from linear algebra.    

\begin{lemma}\label{expander mixing lemma for ERq}
Let $G$ be a polarity graph of a projective plane of order $q$, and let $S$ be a subset of $V(G)$.
Let $e(S)$ denote the number of edges in $S$, possibly including loops.Then 
\[
e(S)\geq \frac{(q+1)|S|^2}{2(q^2+q+1)} - \frac{\sqrt{q}|S|}{2}.
\]
\end{lemma}
\begin{proof}
Let $A$ be the adjacency matrix of $G$ and let $n=q^2+q+1$. Let $\{x_i\}$ be an orthonormal set of eigenvectors of $A$. Since $A$ has constant row sum, $x_1 = \frac{1}{\sqrt{n}} \mathbf{1}$ and $\lambda_1 = q+1$. By Lemma \ref{eigenvalues of polarity graph}, the other eigenvalues of $A$ are all $\pm \sqrt{q}$.
\medskip

Now let $S$ be a subset of $V(G)$ and let $\mathbf{1}_S$ be the characteristic vector for $S$. Let $\hat{e}(S)$ denote the number of non-loop edges of $S$ and $l(S)$ denote the number of loops in $S$. Then
\begin{equation}\label{sum over S}
\mathbf{1}_S^T A \mathbf{1}_S = \sum_{i,j\in S} A_{ij} = 2\hat{e}(S) + l(S).
\end{equation}
Next we give a spectral decomposition of $\mathbf{1}_S$:
\[
\mathbf{1}_S = \sum_{i=1}^n \langle \mathbf{1}_S,x_i\rangle x_i.
\]
Noting that $\langle \mathbf{1}_S, x_1\rangle = \frac{|S|}{\sqrt{n}}$ and expanding \eqref{sum over S}, we see that
\[
2\hat{e}(S) + l(S) = \sum_{i=1}^n \langle \mathbf{1}_S, x_i\rangle^2 \lambda_i = \frac{(q+1)|S|^2}{n} + \sum_{i=2}^n \langle \mathbf{1}_S, x_i\rangle^2 \lambda_i.
\]
Therefore,
\[
\left|2 \hat{e}(S) + l(S) - \frac{(q+1)|S|^2}{n} \right| \leq \sum_{i=2}^n \left|\langle \mathbf{1}_S, x_i\rangle^2 \lambda_i \right| \leq \sqrt{q} \sum_{i=2}^n \langle \mathbf{1}_S, x_i\rangle^2 \leq \sqrt{q} |S|.
\]
Since $e(S) = \hat{e}(S) + l(S)$ and $l(S)\geq 0$, rearranging gives the result.
\end{proof}

\medskip

Note that Lemma \ref{expander mixing lemma for ERq} does not give us any information when $|S| = O( q)$.  Lemma \ref{expander mixing lemma for ERq} is not strong enough for our purposes in terms of proving Corollary \ref{abl conjecture}.


\section{Proof of Theorem \ref{dense subgraph} and Corollary \ref{abl conjecture}}\label{proofs}

In this section we prove Theorem \ref{dense subgraph} and Corollary \ref{abl conjecture}.  

\medskip

\begin{proof}[Proof of Theorem \ref{dense subgraph}]
Let $\Pi$ be a finite projective plane of order $q$ that contains an oval $H$.  Let $\pi$ be a polarity of $\Pi$ and let $G$ be the corresponding polarity graph.  We omit the subscript $\pi$ for notational convenience.  For $v \in V(G)$, write $\Gamma (v)$ for the set of neighbors of $v$ in $G$.  Given $S \subset H$, let 
\[
Y_S = \{ v \in V(G) : | \Gamma (v) \cap S | = 2 \}
\]
and $X_S = Y_S \backslash S$.  Since $H$ is an oval, the number of secants to $H$ is $\binom{q+1}{2}$.  
Thus, for any pair of distinct vertices $s_i , s_j \in H$, there is a unique vertex $t_{i,j} \in Y_H$ such that 
$t_{i,j}$ is adjacent to both $s_i$ and $s_j$.  The vertex $t_{i,j}$ corresponds to the unique secant that intersects $H$ at 
$s_i$ and $s_j$.  Further, the only neighbors of $t_{i,j}$ in $H$ are $s_i$ and $s_j$ and so  
\[
|Y_S | = \binom{ |S| }{2}
\] 
for any $S \subset H$.  This implies 
\begin{equation}\label{2:eq0}
|S| + |X_S| \geq |Y_S| = \binom{ |S| }{2}.
\end{equation}
When $S = H$, we get that $|X_H| \geq \binom{q+1}{2}  - (q  + 1)$ so by Lemma \ref{expander mixing lemma for ERq}, 
\begin{equation}\label{22:eq1}
e( G[ X_H] ) \geq \frac{q^3}{8} - O (q^{5/2}).  
\end{equation}
Let $m \in \{ 1,2 , \dots ,  q + 1 \}$.  Choose $S \subset H$ uniformly at random from the set of all subsets of $H$ of size $m$.  If 
$e(S , X_S)$ is the number of edges with one endpoint in $S$ and the other in $X_S$, then using (\ref{2:eq0}), 
\begin{equation}\label{2:eq1}
e(S , X_S) = 2 |X_S| \geq 2 \binom{ |S| }{2} - 2 |S| = 2 \binom{m}{2} - 2m.
\end{equation}
If $e = uv \in E( G[X_H] )$, then the at most four vertices in $( \Gamma (u) \cap H ) \cup ( \Gamma (v) \cap H )$ must be chosen in $S$ in order to have $e \in E( G [ X_S] )$.  Therefore, 
\[
\mathbb{P} ( e \in E (G[  X_S ] ) ) \geq \frac{ \binom{q-3}{m-4} }{ \binom{q+1}{m}} = \frac{ m ( m - 1) ( m - 2)( m - 3) }{ (q + 1) q ( q - 1) ( q- 2) } \geq \frac{ m ( m -1) ( m -2) ( m - 3) }{q^4}.
\]
By (\ref{22:eq1}) and linearity of expectation, 
\begin{equation}\label{2:eq2}
\mathbb{E} ( e (G [ X_S] ) ) \geq \frac{ m ( m - 1)(m-2)(m-3) }{8q} - O \left( \frac{m^4}{q^{3/2}} \right).
\end{equation}
Combining (\ref{2:eq1}) and (\ref{2:eq2}), we see that there is a choice of $S \subset H$ with $|S| = m$ and 
\[
e ( G[ S \cup X_S ] ) \geq 2 \binom{ m }{2} + \frac{ m ( m - 1)(m-2)(m-3) }{8q} - O \left( \frac{m^4}{q^{3/2}}+m \right).
\]
Lastly, observe $|S \cup X_S| \leq |S| + |Y_S| = m + \binom{m}{2}$.  
\end{proof}

\medskip

Now we use Theorem \ref{dense subgraph} to prove Corollary \ref{abl conjecture}.  We first prove a simple inequality that
expresses the number of edges of an induced subgraph of a polarity graph in terms of the removed set of vertices.

Let $G$ be a polarity graph of a projective plane of order $q$ and let $X\subset V(G)$. 
The number of edges in the graph $G \setminus X$ is 
\[
e(G) - e(X) - e(X, X^c)
\]
where $e(X)$ includes counting loops in $G$.  
Since
\[
e(X) + e(X, X^c) = \sum_{x\in X} d(x) - e(X) \leq (q+1)|X| - e(X),
\]
we have
\begin{equation}\label{2:eq3}
e(G \setminus X ) \geq e(G) - (q + 1) |X| + e(X).
\end{equation}

\medskip

\begin{proof}[Proof of Corollary \ref{abl conjecture}]
Let $q$ be a prime power and $ER_q$ be the Erd\H{o}s-R\'{e}nyi orthogonal polarity graph.  It is known that this graph has 
$\frac{1}{2}q(q+1)^2$ edges.  Let $m$ be the largest integer satisfying $m+\binom{m}{2} \leq 2q+3$. 
Then $m = \lfloor \sqrt{4q + 25/4}  - 1/2 \rfloor$ and 
\[
2 \binom{m}{2} + \frac{ m (m - 1)(m-2)(m-3) }{q^4} = 6q - O(q^{1/2} ).
\]
By Theorem \ref{dense subgraph}, there is a set $S \subset V(ER_q)$ with $|S| = m + \binom{m}{2}$ such that 
$S$ induces a subgraph with at least $6q - O (q^{1/2})$ edges.  Let $X = S \cup S'$ where $S'$ is an arbitrarily chosen set of 
$2q + 3 - |S|$ vertices disjoint from $S$.  Then by (\ref{2:eq3}),
\[
e( ER_q \setminus X) \geq \frac{1}{2}q(q+1)^3 - (q+1)(2q +3) + 6q - O(q^{1/2} ) =  \frac{1}{2}q^3 - q^2 + \frac{3}{2}q - O\left(q^{1/2}\right).
\]
Since $ER_q$ is $C_4$-free, we have 
\[
\textup{ex}( q^2 - q - 2 , C_4)  \geq \frac{1}{2}q^3 - q^2 + \frac{3}{2}q - O\left(q^{1/2}\right).
\]
\end{proof}


\section{Concluding remarks}\label{conclusion}

There are two special circumstances in which one can improve Theorem \ref{dense subgraph}. Each indicates the difficulty of finding exact values for the parameter $\mathrm{ex}(n, C_4)$.

\begin{itemize}
\item The first situation is when $q$ is a square.  In this case, $\mathbb{F}_{q}$ contains the subfield $\mathbb{F}_{\sqrt{q}}$ and this subfield may be used to find small graphs that contain many edges.  For instance $ER_{q}$ contains a subgraph $F$ that is isomorphic to $ER_{ \sqrt{q} }$.
One can choose $m = \sqrt{q} + 1$ and let $S$ be the set of absolute points in $F$.  These $m$ vertices will also be absolute points in $ER_q$ and thus are contained in an oval (the absolute points of an orthogonal polarity of $PG(2,q)$ form an oval when $q$ is odd).  If we then consider the $\binom{m}{2}$ vertices in $Y_S$, these will be the vertices in $F$ that are adjacent to the absolute points of $F$.  The set $Y_S$ induces a $\frac{1}{2}(\sqrt{q} -1)$-regular graph in $F$ (see \cite{bs}).  
The set $X = S \cup Y_S$ will span roughly $\frac{q^{3/2}}{8}$ edges which is much larger than the linear in $q$ lower bound provided by Theorem \ref{dense subgraph} when $m  = \sqrt{q} + 1$.  

\item The second situation is when $q$ is a power of 2 and $q  - 1$ is prime.   
Assume that this is the case and consider $ER_{q-1}$.  Let $F$ be a subgraph of 
$ER_{q-1}$ obtained by deleting three vertices of degree $q-1$.  The number of vertices of $F$ is 
$(q - 1)^2 + (q - 1) + 1 - 3 = q^2 - q - 2$, and the number of edges of $F$ is at least 
$\frac{1}{2} (q - 1) q^2 - 3(q  - 1)  = \frac{1}{2}q^3 - \frac{1}{2}q^2 - 3q + 3$.  This is larger than $\frac{1}{2}q^3 - q^2$ whenever $q \geq 5$.  A prime of the form $2^m-1$ with $m \in \mathbb{N}$ is known as a Mersenne Prime.  It is an open problem to determine if there are infinitely many Mersenne Primes.  It has been conjectured that there are infinitely many such primes.  
\end{itemize}

In \cite{tt}, Sidon sets are used to construct $C_4$-free graphs.  For a prime power $q$, these graphs have $q^2-1$ vertices, and $\frac{1}{2}q^3 - q + \frac{1}{2}$ edges when $q$ is odd, and $\frac{1}{2}q^3 - q$ edges when $q$ is even.  These graphs have a degree sequence similar to the degree sequence of an orthgonal polarity graph and it seemed possible that these graphs could be extremal.  However, Theorem \ref{dense subgraph} can be applied to show
\[
\mathrm{ex}(q^2-1, C_4) \geq \frac{1}{2}q^3 - O(\sqrt{q}),
\]

\noindent which shows that the graphs constructed in \cite{tt} are not extremal.  


\end{document}